\documentclass[preprint,12pt]{elsarticle}

\usepackage{lineno,hyperref}
\usepackage{amsmath}
\usepackage{amsthm}
\usepackage{amssymb}
\usepackage{graphicx}
\newtheorem{lemma}{Lemma}
\newtheorem{theorem}{Theorem}
\newtheorem{remark}{Remark}
\newtheorem{corollary}{Corollary}
\newtheorem{definition}{Definition}

\modulolinenumbers[5]

\journal{Communications in nonlinear analysis and ...}

\begin{document}

\begin{frontmatter}

\title{Existence and uniqueness of global solutions of fractional functional differential equations with bounded delay}

\author[Firstaddress]{Chung-Sik Sin\corref{mycorrespondingauthor}}
\cortext[mycorrespondingauthor]{Corresponding author}
 \ead{chongsik@163.com}

\address[Firstaddress]{Faculty of Mathematics, \textit{\textbf {Kim Il Sung}} University, Kumsong Street, Taesong District, Pyongyang, D.P.R.KOREA}

\begin{abstract}
This paper deals with initial value problems for fractional functional differential equations with bounded delay.
The fractional derivative is defined in the Caputo sense.
By using the Schauder fixed point theorem and the properties of the Mittag-Leffler function, new existence and uniqueness results for global solutions of the initial value problems are established.
In particular, the unique existence of global solution is proved under the condition close to the Nagumo-type condition.
\end{abstract}


\end{frontmatter}

%% \linenumbers

%% main text
\section{Introduction}\label{Sec:1}

The present paper considers initial value problems for fractional functional differential equations of the form
\begin{equation}
\label{governing_equation}
^CD^\alpha u(t) =f(t,u_t),
\end{equation}
subject to initial conditions
\begin{equation}
\label{initial_condition}
u^{(i)}(t)=\phi^{(i)}(t) \text{ on } [-h,0], i=0,1,...,\lceil \alpha \rceil-1,
\end{equation}
where $ \alpha,h>0 $, $ \phi \in C^{\lceil \alpha \rceil-1}[-h,0] $, the symbol $ ^cD^\alpha $ denotes the Caputo fractional differential operator and $ u_t $ is defined by $ u_t(\theta)=u(t+\theta) \text{ for } \theta \in [-h,0] $.
Here $\lceil \alpha \rceil$ is the smallest integer $ \geq \alpha $.

Fractional calculus has been widely used to describe the complex nonlinear phenomena in continuum mechanics, thermodynamics, quantum mechanics, plasma dynamics, electrodynamics and so on \cite{Uchaikin}. The constitutive relation of the viscoelastic media is successfully modeled by fractional differential equations \cite{Mainardi_book}. Wang et al. \cite{WangHuang} proposed a fractional order financial system with time delay and considered the dynamical behaviors of such a system.

Existence and uniqueness of solutions of initial value problems for fractional ordinary differential equations have been investigated by many authors \cite{DieFor,DieBoo,DieNag,KilSri,Lin,LakLee,Sin}.
Diethelm and Ford \cite{DieFor} used the Schauder fixed point theorem to prove that the initial value problem for the fractional differential equation with Caputo derivative has a local solution under appropriate assumptions.
Lakshmikantham \cite{LakLee} first extended the Nagumo-type existence and uniqueness result for fractional differential equations involving Riemann-Liouville fractional derivative.
In \cite{DieNag} the classical Nagumo-type theorem is generalized to Caputo-type fractional differential equations.
Lin \cite{Lin} obtained existence results for solutions of the initial value problems under more general assumptions.
Recently, Sin et al. \cite{Sin} improved sufficient conditions for existence and uniqueness of global solutions of the initial value problem by proving a new property of the two parameter Mittag-Leffler function.

With the development of mathematical theory on fractional ordinary differential equations, fractional functional differential equations with delay have also been intensively studied \cite{Benchohra,Lakshmikantham,ZhouJiao,AgarwalZhou,Zhou,YangCao,WangXiao}.
Lakshmikantham \cite{Lakshmikantham} established existence results for solutions of initial value problems for fractional functional differential equations with bounded delay.
Agarwal et al. \cite{AgarwalZhou} used Krasnoselskii¡Çsfixed point theorem to prove sufficient conditions for existence of solutions of fractional neutral functional differential equations with bounded delay.
Zhou et al. \cite{ZhouJiao} obtained various criteria on existence and uniqueness of solutions of fractional neutral functional differential equations with infinite delay.
Yang et al. \cite{YangCao} considered existence and uniqueness of local and global solutions of the initial value problem (\ref{governing_equation})-(\ref{initial_condition}). 
In order to obtain the existence result for global solutions, they supposed the following condition:
 there exist constants $ c_1, c_2 \geq 0 $ and $ 0<\lambda<1 $ such that
\begin{equation}
\label{previous_condition}
|f(x,u_t)|\leq c_1+c_2 \|u_t\|_{C[-h,0]}^{\lambda}.
\end{equation} 
As pointed in \cite{YangCao}, the condition is violated even by some very elementary equations like linear equation.
They proved the unique existence of solution when $ f $ is continuous and satisfies the Lipschitz condition with respect to the second variable.

In this paper, by using Schauder fixed point theorem and properties of two parameter Mittag-Leffler function, existence and uniqueness results for global solutions of the equation (\ref{governing_equation})-(\ref{initial_condition}) are improved. Firstly, the condition $ 0<\lambda<1 $ for global solutions is replaced with the more general one $ 0<\lambda \leq 1 $.  Then the unique existence of global solutions is established when $ f $ satisfies the condition close to the Nagumo-type condition weaker than the Lipschitz condition.

The rest of the paper is organized as follows.
Section \ref{Sec:2} introduces the preliminary results which are used in deriving the main results of this paper.
In Section \ref{Sec:3}, the existence of solutions of the initial value problem (\ref{governing_equation})-(\ref{initial_condition}) is discussed.
Section \ref{Sec:4} deals with the uniqueness of global solutions of the equation.

%------------------------------------------------------------------------------------------------------------------------
\section{Preliminaries}\label{Sec:2}
In this section we give definitions and preliminary results which are needed in our investigation.
Firstly, we recall the concepts of Riemann-Liouville fractional integral operator and Caputo fractional differential operator.
\begin{definition}[\cite{DieBoo}]
Let $ \beta,l \geq 0 $ and $ u \in L^{1}[0,l] $.
The Riemann-Liouville integral of order $ \beta $ of $ u $ is defined by
\begin{equation}
\nonumber
 I^\beta u(t)={\frac{1}{\Gamma (\beta) } \int^{t}_{0}{(t-s)}^{\beta-1} u(s) ds}.
\end{equation}
\end{definition}

\begin{definition}[\cite{DieBoo}]
Let $ \beta,l\geq 0 $ and $ D^{\lceil \beta \rceil}u \in L^{1}[0,l] $. The Caputo fractional derivative of order $ \beta $ of $ u $  is defined by
\begin{equation}
 \nonumber
^CD^\beta u(t)=I^{\lceil \beta \rceil-\beta}D^{\lceil \beta \rceil}u(t).
 \end{equation}
\end{definition}
Then the Caputo fractional derivative of order $ \beta $  is defined for $ u $ belonging to the space $ AC^{\lceil \beta \rceil}[0,l] $ of absolutely continuous functions of order $\lceil \beta \rceil$.
\begin{definition}
Let $ T>0 $. A function $ u:[-h,T] \rightarrow R $  is called a solution of the initial value problem (\ref{governing_equation})-(\ref{initial_condition}) on $ [0,T]$ if
$ u|_{[-h,0]}=\phi$, $ u|_{[0,T]} \in AC^{\lceil \alpha \rceil}[0,T] $ and $ u $ satisfies the equation (\ref{governing_equation}) for $ t \in [0,T] $.
\end{definition}
Here $ u|_{[0,T]} $  means the restriction of the function $ u $ to $[0,T]$.
For the next lemmas, we make the following assumption:\\
(H2-1) $ T>0 $ and the function $f :[0,T]\times C[-h,0] \rightarrow R $ is  continuous.
\begin{lemma}
Suppose that (H2-1) holds and $ u:[-h,T] \rightarrow R $ satisfies the conditions such that $ u|_{[-h,0]}=\phi$ and $ u|_{[0,T]} \in AC^{\lceil \alpha \rceil}[0,T] $. Then the function $ u $ is a solution of the initial value problem (\ref{governing_equation})-(\ref{initial_condition}) on $ [0,T]$ if and only if $ u|_{[0,T]} $ is a solution of the integral equation
\begin{equation}\label{equivalent_integral_equation}
u(t)=\sum^{\lceil \alpha \rceil-1}_{i=0}\frac{t^i}{i!} \phi^{(i)}(0)+\frac{1}{\Gamma (\alpha) }\int_0^t (t-s)^{\alpha-1}f(s,u_s)ds .
\end{equation}
\end{lemma}
\begin{proof}
See Lemma 2.8 in \cite{YangCao}.
\end{proof}
We can easily see that if $f$ is  continuous and $ u  $ is the solution of the integral equation (\ref{equivalent_integral_equation})  on $ [0,T]$, then $ u|_{[0,T]} \in AC^{\lceil \alpha \rceil}[0,T] $.
Thus, in order to obtain the solution of the initial value problem (\ref{governing_equation})-(\ref{initial_condition}), we will find the solution of the integral equation (\ref{equivalent_integral_equation}) in $ Y $, where
$ Y $ means the Banach space consisting of all continuous functions belonging in $C[0,T]$ with the Chebyshev norm $ \|\cdot\|_{C[0,T]} $.
Then the integral equation (\ref{equivalent_integral_equation}) is transformed into the fixed point problem with the operator $ J:Y \rightarrow Y $ defined by
\begin{equation}\label{fixed point problem}
Ju(t)=\sum^{\lceil \alpha \rceil-1}_{i=0}\frac{t^i}{i!} \phi^{(i)}(0)+\frac{1}{\Gamma (\alpha) }\int_0^t (t-s)^{\alpha-1}f(s,u_s)ds.
\end{equation}

\begin{lemma}\label{operator_continuity}
Let $B \subset Y $ be a bounded set. If (H2-1) holds, then the operator $ J $ is continuous in $ B $.
\end{lemma}
\begin{proof}
See Theorem 3.1 in \cite{YangCao}.
\end{proof}

\begin{lemma}\label{equicontinuity}
Let $B \subset Y $ be a bounded set. If (H2-1) holds, then the set $ J(B) $ is equicontinuous.
\end{lemma}
\begin{proof}
See Theorem 3.1 in \cite{YangCao}.
\end{proof}

\begin{lemma}\label{fixed point}
Let $B \subset Y $ be a bounded and convex set. If (H2-1) holds and $J(B)\subset B$, then $ J $ has a fixed point in $B$.
\end{lemma}
\begin{proof}
By Lemma \ref{operator_continuity}, Lemma \ref{equicontinuity} and  Arzela-Ascoli theorem, the operator $ J:B \rightarrow B $ is compact.
Then by Schauder fixed point theorem, $ J$ has at least one fixed point in $B$.
\end{proof}

The Mittag-Leffler functions play a crucial role in our investigation.

\begin{definition}[\cite{DieBoo}]
 Let $ c,d>0 $. A two-parameter function of the Mittag-Leffler type is defined by the series expansion
 \begin{equation}\label{eq2-1}
 E_{c,d}(t)={\sum^\infty_{i=0}{\frac{t^i}{\Gamma (ci+d)}}}.
 \end{equation}
 \end{definition}

 \begin{lemma}[\cite{Sin}]
 \label{Mittag_function_property}
Let $ c, l, r, \beta>0$. If $d<min\{\beta,1\}$, then there exists a real number $ \lambda>0 $ such that
\begin{equation}\label{eq2-6}
\frac{1}{\Gamma (\beta)}\int^{t}_{0}{(t-s)}^{\beta-1}{\frac {E_{c,1-d}({\lambda s^c})}{s^d}ds }<rE_{c,1-d}({\lambda t^c}), t \in [0,l].
\end{equation}
\end{lemma}
%------------------------------------------------------------------------------------------------------------------------

\section{Existence of global solutions}\label{Sec:3}
In this section the existence of global solutions of the initial value problem for the fractional functional differential equation (\ref{governing_equation})-(\ref{initial_condition}) is investigated.
For the next theorem, we make the following assumption:\\
(H3-1) there exist $T_1 \in (0,T]$, $ \eta \in (T_1,T) $, $a_1,a_2,q_1,q_2 \in [0,\alpha)$, $ p_1, p_2 \in (0,1] $, $ b_1, b_2>0 $, $ m_1(t) \in L^{\frac{1}{q_1}} [0,T_1] $,
$ m_2(t) \in L^{\frac{1}{q_2}} [T_1,T] $ such that
\begin{equation}
\nonumber
 |f(t,u)| \leq \left\{
                \begin{aligned}
                  & \frac{b_1}{t^{a_1}}\|u\|_{C[-h,0]}^{p_1}+m_1(t) && \text{for $ t \in (0, T_1] $ and $ u \in C[-h,0] $} \\
                  & \frac{b_2}{|t-\eta|^{a_2}}\|u\|_{C[-h,0]}^{p_2}+ m_2(t) && \text{for $ t \in [T_1,\eta) \cup(\eta, T] $ and $ u \in C[-h,0] $}.\\
                \end{aligned}
              \right.
\end{equation}

\begin{theorem}
\label{existence_theorem}
Suppose that (H2-1) and (H3-1) hold. Then the initial value problem (\ref{governing_equation})-(\ref{initial_condition}) has a  solution  on $ [0,T]$.
\end{theorem}
\begin{proof}
By Lemma \ref{Mittag_function_property},  there exist $ \lambda_1,\lambda_2,\lambda_3>0 $ such that
for $ t \in [0,T] $,
\begin{align}
\nonumber
&\{I^\alpha[E_{1,1-a_1}(\lambda_1 s)s^{-a_1}]\}(t)<\frac{1}{2b_1}E_{1,1-a_1}(\lambda_1 t),\\
\nonumber
&\{I^{\alpha-a_2}[E_{1,1}(\lambda_2 s)]\}(t)<\frac{\Gamma(\alpha)}{2b_2\Gamma(\alpha-a_2)}E_{1,1}(\lambda_2 t),\\
\nonumber
&\{I^\alpha[E_{1,1-a_2}(\lambda_3 s)s^{-a_2}]\}(t)<\frac{1}{2b_2}E_{1,1-a_2}(\lambda_3 t).
\end{align}
We define a convex bounded closed subset $ G $ of $ Y $ as follows\\
\begin{equation}
\nonumber
G=\Biggl \{ x \in C[0,T] :|x(t)| \leq \left\{
                            \begin{aligned}
                                & 2D_1E_{1,1-a_1}(\lambda_1 t) && \text{for $ t \in [0, T_1] $} \\
                                & 2D_2E_{1,1}(\lambda_2 t) && \text{for $ t \in [T_1, \eta] $}\\
                                & 2D_3E_{1,1-a_2}(\lambda_3 (t-\eta)) && \text{for $ t \in [\eta, T] $} \\
                            \end{aligned}
                            \right.
\Biggr \},
\end{equation}
where
\begin{align}
\nonumber
&D_1= \Gamma (1-a_1) max \Bigg\{1,\|\phi\|_{C[-h,0]}, \sum^{\lceil \alpha \rceil-1}_{i=0}\frac{T^i}{i!} |\phi^{(i)}(0)|
+\frac{{T_1}^{\alpha-q_1} }{\Gamma(\alpha)}\bigg(\frac{1-q_1}{\alpha-q_1}\bigg)^{1-q_1}\\
\nonumber
&\|m_1\|_{L^{\frac{1}{q_1}}[0,T_1]}\Bigg\},\\
\nonumber
&D_2= 2D_1E_{1,1-a_1}(\lambda_1 {T_1})+\frac{1}{\Gamma(\alpha)}\bigg(\frac{1-q_2}{\alpha-q_2}\bigg)^{1-q_2}
{(\eta-T_1)}^{\alpha-q_2} \|m_2\|_{L^{\frac{1}{q_2}}[T_1,\eta]},\\
\nonumber
&D_3=\Gamma (1-a_2)\Bigg[2D_2E_{1,1}(\lambda_2 \eta) +\frac{1}{\Gamma(\alpha)}\bigg(\frac{1-q_2}{\alpha-q_2}\bigg)^{1-q_2}
(T-\eta)^{\alpha-q_2} \|m_2\|_{L^{\frac{1}{q_2}}[\eta,T]}\Bigg].
\end{align}
By Holder inequality, we have that for any $ x \in G $ and $ t \in [0,T_1] $,
\begin{align}
\nonumber
&|Jx(t)|\leq \sum^{\lceil \alpha \rceil-1}_{i=0}\frac{t^i}{i!} |\phi^{(i)}(0)|+\frac{1}{\Gamma(\alpha)}\int_0^t (t-s)^{\alpha-1}|f(s,x_s)|ds\\
\nonumber
&\leq\sum^{\lceil \alpha \rceil-1}_{i=0}\frac{t^i}{i!} |\phi^{(i)}(0)|+\frac{1}{\Gamma(\alpha)}\int_0^t (t-s)^{\alpha-1} \bigg(\frac{b_1}{s^{a_1}}\|x_s\|_{C[-h,0]}^{p_1}+m_1(s)\bigg)ds\\
\nonumber
&\leq \sum^{\lceil \alpha \rceil-1}_{i=0}\frac{t^i}{i!} |\phi^{(i)}(0)|+\frac{2b_1 D_1}{\Gamma(\alpha)}\int_0^t (t-s)^{\alpha-1} \frac{E_{1,1-a_1}(\lambda_1 s)}{s^{a_1}}ds\\
\nonumber
&+\frac{1}{\Gamma(\alpha)}\bigg(\int_0^t (t-s)^{\frac{\alpha-1}{1-q_1}}ds\bigg)^{1-q_1}\|m_1\|_{L^{\frac{1}{q_1}}[0,T_1]} \\
\nonumber
&\leq \frac{D_1}{\Gamma(1-a_1)}+D_1E_{1,1-a_1}(\lambda_1 t)< 2D_1E_{1,1-a_1}(\lambda_1 t).
\end{align}
We have that for any $ x \in G $ and $ t \in [T_1,\eta] $,
\begin{align}
\nonumber
&|Jx(t)|\leq \sum^{\lceil \alpha \rceil-1}_{i=0}\frac{t^i}{i!} |\phi^{(i)}(0)|+\int_0^{T_1} \frac{(t-s)^{\alpha-1}}{\Gamma(\alpha)}|f(s,x_s)|ds +\int_{T_1}^t \frac{(t-s)^{\alpha-1}}{\Gamma(\alpha)}|f(s,x_s)|ds
\end{align}
\begin{align}
\nonumber
&\leq 2D_1E_{1,1-a_1}(\lambda_1 {T_1})+\frac{1}{\Gamma(\alpha)}\int_{T_1}^t (t-s)^{\alpha-1}\bigg(\frac{b_2}{(\eta-s)^{a_2}}\|x_s\|^{p_2}+ m_2(s)\bigg)ds\\
\nonumber
&\leq 2D_1E_{1,1-a_1}(\lambda_1 {T_1})+\int_{T_1}^t \frac{(t-s)^{\alpha-1}}{\Gamma(\alpha)}m_2(s)ds+\int_{T_1}^t \frac{b_2(t-s)^{\alpha-a_2-1}}{\Gamma(\alpha)}\|x_s\|_{C[-h,0]}^{p_2}ds\\
\nonumber
&\leq 2D_1E_{1,1-a_1}(\lambda_1 {T_1})+\frac{1}{\Gamma(\alpha)}\bigg(\int_{T_1}^t (t-s)^{\frac{\alpha-1}{1-q_2}}ds \bigg)^{1-q_2}\|m_2\|_{L^{\frac{1}{q_2}}[T_1,\eta]}\\
\nonumber
&+\frac{2b_2\Gamma(\alpha-a_2)}{\Gamma(\alpha)}D_2\{I^{\alpha-a_2}[E_{1,1}(\lambda_2 s)]\}(t)
\leq  D_2+D_2E_{1,1}(\lambda_2 t)\leq 2D_2E_{1,1}(\lambda_2 t).
\end{align}
We have that for any $ x \in G $ and $ t \in [\eta,T] $,
\begin{align}
\nonumber
&|Jx(t)|\leq \sum^{\lceil \alpha \rceil-1}_{i=0}\frac{t^i}{i!} |\phi^{(i)}(0)|+\int_0^{\eta} \frac{(t-s)^{\alpha-1}}{\Gamma(\alpha)}|f(s,x_s)|ds +\int_{\eta}^t \frac{(t-s)^{\alpha-1}}{\Gamma(\alpha)}|f(s,x_s)|ds\\
\nonumber
&\leq 2D_2E_{1,1}(\lambda_2 \eta)+\frac{1}{\Gamma(\alpha)}\int_\eta^t (t-s)^{\alpha-1}\bigg(\frac{b_2}{(s-\eta)^{a_2}}\|x_s\|_{C[-h,0]}^{p_2}
+ m_2(s)\bigg)ds\\
\nonumber
&\leq 2D_2E_{1,1}(\lambda_2 \eta)+\int_\eta^t \frac{(t-s)^{\alpha-1}}{\Gamma(\alpha)}m_2(s)ds+ \int_0^{t-\eta}\frac{(t-\eta-s)^{\alpha-1}}{\Gamma(\alpha)} \frac{b_2}{s^{a_2}}\|x_{s+\eta}\|_{C[-h,0]}^{p_2}ds\\
\nonumber
&\leq 2D_2E_{1,1}(\lambda_2 \eta)+\frac{1}{\Gamma(\alpha)}\bigg(\int_\eta^t (t-s)^{\frac{\alpha-1}{1-q_2}}ds\bigg)^{1-q_2}
\|m_2\|_{L^{\frac{1}{q_2}}[\eta,T]}\\
\nonumber
&+2b_2D_3\{I^\alpha[E_{1,1-a_2}(\lambda_3 s)]\}(t-\eta)\leq \frac{D_3}{\Gamma(1-a_2)}+D_3E_{1,1-a_2}(\lambda_3 (t-\eta))\\
\nonumber
&\leq 2D_3E_{1,1-a_2}(\lambda_3 (t-\eta)).
\end{align}
Thus $ JG \subset G $.
By Lemma \ref{fixed point}, $ J $  has at least one fixed point in $ G $.
\end{proof}

\begin{remark}
The condition (H3-1) of Theorem \ref{existence_theorem} can be replaced by the following condition.
There exist $ n\in N $, $0=T_0<T_1<\cdots<T_n=T $, $ \eta_i \in (T_{i-1},T_i)$ for $i=2,\cdots,n $, $a_j,q_j \in [0,\alpha)$, $ p_j \in (0,1] $, $ b_j>0 $, $ m_j(t) \in L^{\frac{1}{q_j}} [T_{j-1},T_j] $ for $ j=1,\cdots,n $ such that
\begin{equation}
\nonumber
 |f(t,u)| \leq \left\{
                \begin{aligned}
                  & \frac{b_1}{t^{a_1}}\|u\|_{C[-h,0]}^{p_1}+m_1(t) && \text{for $ t \in (0, T_1] $ and $ u \in C[-h,0] $} \\
                  & \frac{b_i}{|t-\eta_i|^{a_i}}\|u\|_{C[-h,0]}^{p_i}+ m_i(t) && \text{for
                   $ t \in [T_{i-1},\eta_i) \cup(\eta_i, T_i] $ and $ u \in C[-h,0] $},\\
                \end{aligned}
              \right.
\end{equation}
where $ i=2,\cdots,n $.
\end{remark}

\begin{remark}
Theorem \ref{existence_theorem} is an improvement on the result of Corollary 3.1 in \cite{YangCao}.
\end{remark}

\begin{remark}
Theorem \ref{existence_theorem} can be easily extended to vector-valued functions.
\end{remark}

\begin{corollary}
\label{global_existence}
Suppose that (H2-1) and (H3-1) hold, except that the number $ T $ is taken to be $ \infty $.
Then the initial value problem (\ref{governing_equation})-(\ref{initial_condition}) has a solution  on $ [0, \infty)$.
\end{corollary}
\begin{proof}
By Theorem \ref{existence_theorem}, for any $ T \in R $,  the fractional functional differential equation (\ref{governing_equation})-(\ref{initial_condition}) has a solution.
Since $ T $ can be chosen arbitrarily large, the equation (\ref{governing_equation})-(\ref{initial_condition}) has at least one global solution on $ [0, \infty)$.
\end{proof}

% ------------------------------------------------------------------------

\section{Uniqueness of global solutions}\label{Sec:4}
In this section the uniqueness of global solutions of the initial value problem for the fractional functional differential equation (\ref{governing_equation})-(\ref{initial_condition}) is studied.
For the uniqueness theorem, we make the following hypotheses:\\
(H4-1) there exist constants $ a_1,a_2 \in [0,\alpha), b_1, b_2>0 $, $ T_1 \in (0,T]$,
$ \eta \in (T_1,T)$ such that
\vskip -11pt
\begin{equation}
\nonumber
|f(t,u)-f(t,v)|\leq \left\{ \begin{aligned}
                             & \frac {b_1}{t^{a_1}}\|u-v\|_{C[-h,0]} && \text {for $t \in (0,T_1]$, $ u,v \in C[-h,0] $}\\
                             & \frac {b_2}{|t-\eta|^{a_2}}\|u-v\|_{C[-h,0]} && \text { for $ t \in [T_1,\eta)\cup (\eta,T]$, $ u,v \in C[-h,0] $}.
                    \end{aligned}
                    \right.
\end{equation}

\begin{theorem}\label{uniqueness_theorem}
Suppose that (H2-1) and (H4-1) hold. Then the  initial value problem  (\ref{governing_equation})-(\ref{initial_condition}) has a unique solution on $ [0,T]$.
\end{theorem}
\begin{proof}
It follows from the condition (H4-1) that (H3-1) holds.
Thus, by Theorem \ref{existence_theorem}, the fractional functional differential equation  (\ref{governing_equation})-(\ref{initial_condition}) has at least one solution on $ [0, T]$.
By using the method of proof by contradiction, we will obtain a uniqueness result for solutions of the initial value problem (\ref{governing_equation})-(\ref{initial_condition}).
Assume that the equation (\ref{governing_equation})-(\ref{initial_condition}) has two solutions on $ [0, T]$.
Then the operator $ J $ has also two fixed points $x,y \in C[0,T]$ such that $\|x-y\|_{C[0,T]}>0$.
By Lemma \ref{Mittag_function_property}, there exists a real number $ \lambda_1>0 $
such that for $ t \in [0,T] $,
\begin{equation}
\nonumber
\{I^\alpha[E_{1,1-a_1}(\lambda_1 s)s^{-a_1}]\}(t)<\frac{1}{b_1}E_{1,1-a_1}(\lambda_1 t).
\end{equation}
We define $W_1$ and $Q_1$ as follows
\begin{align}
\nonumber
W_1&=\inf\{w:|x(t)-y(t)|\leq wE_{1,1-a_1}(\lambda_1t),t\in [0,T_1]\},\\
\nonumber
Q_1&=\inf\{t\in[0,T_1]:|x(t)-y(t)|=W_1E_{1,1-a_1}(\lambda_1t)\}.
\end{align}
If $ W_1 \neq 0 $, then we have
\begin{align}
\nonumber
W_1E_{1,1-a_1}&(\lambda_1{Q_1})=|x(Q_1)-y(Q_1)|\\
\nonumber
&\leq \frac{1}{\Gamma(\alpha)}\int_0^{Q_1} (Q_1-s)^{\alpha-1}|f(s,x_s)-f(s,y_s)|ds\\
\nonumber
&\leq \frac{1}{\Gamma(\alpha)}\int_0^{Q_1} (Q_1-s)^{\alpha-1}\frac{b_1}{s^{a_1}}\|x_s-y_s\|_{C[-h,0]}ds\\
\nonumber
&\leq \frac{b_1}{\Gamma(\alpha)}\int_0^{Q_1} (Q_1-s)^{\alpha-1} \frac{1}{s^{a_1}}W_1E_{1,1-a_1}(\lambda_1s)ds\\
\nonumber
&< W_1E_{1,1-a_1}(\lambda_1{Q_1}),
\end{align}
which implies that $ W_1=0 $. Therefore $ x(t)=y(t), t\in [0,T_1] $.
By Lemma \ref{Mittag_function_property}, there exists a real number $ \lambda_2>0 $
such that for $ t \in [0,T] $,
\begin{equation}
\nonumber
\{I^{\alpha-a_2}[E_{1,1}(\lambda_2 s)]\}(t)<\frac{\Gamma(\alpha)}{b_2\Gamma(\alpha-a_2)}E_{1,1}(\lambda_2 t).
\end{equation}
We define $W_2$ and $Q_2$ as follows
\begin{align}
\nonumber
W_2&=\inf\{w:|x(t)-y(t)|\leq wE_{1,1}(\lambda_2t),t\in [T_1,\eta]\},\\
\nonumber
Q_2&=\inf\{t\in[T_1,\eta]:|x(t)-y(t)|=W_2E_{1,1}(\lambda_2t)\}.
\end{align}
If $ W_2 \neq 0 $, then we have
\begin{align}
\nonumber
W_2&E_{1,1}(\lambda_2{Q_2})=|x(Q_2)-y(Q_2)|\\
\nonumber
&\leq \frac{1}{\Gamma(\alpha)}\int_0^{Q_2} (Q_2-s)^{\alpha-1}|f(s,x_s)-f(s,y_s)|ds\\
\nonumber
&\leq \frac{1}{\Gamma(\alpha)}\int_{T_1}^{Q_2} (Q_2-s)^{\alpha-1}\frac {b_2}{(\eta-s)^{a_2}}\|x_s-y_s\|_{C[-h,0]}ds\\
\nonumber
&\leq \frac{b_2}{\Gamma(\alpha)}\int_0^{Q_2} (Q_2-s)^{\alpha-a_2-1}W_2E_{1,1}(\lambda_2s)ds
<W_2E_{1,1}(\lambda_2{Q_2}),
\end{align}
which implies that $ W_2=0 $. Therefore $ x(t)=y(t), t\in [0,\eta] $.
By Lemma \ref{Mittag_function_property}, there exists a real number $ \lambda_3>0 $
such that for $ t \in [0,T] $,
\begin{equation}
\nonumber
\{I^q[E_{1,1-a_2}(\lambda_3 s)s^{-a_2}]\}(t)<\frac{1}{b_2}E_{1,1-a_2}(\lambda_3 t).
\end{equation}
We define $W_3$ and $Q_3$ as follows
\begin{align}
\nonumber
W_3&=\inf\{w:|x(t)-y(t)|\leq wE_{1,1-a_2}(\lambda_3(t-\eta)),t\in [\eta,T]\},\\
\nonumber
Q_3&=\inf\{t\in[\eta,T]:|x(t)-y(t)|=W_3E_{1,1-a_2}(\lambda_3(t-\eta))\}.
\end{align}
From the assumption $\|x-y\|_{C[0,T]}>0$, it is clear that $ W_3 \neq 0 $.
Then we have
\begin{align}
\nonumber
W_3&E_{1,1-a_2}(\lambda_3 (Q_3-\eta))=|x(Q_3)-y(Q_3)|\\
\nonumber
&\leq  \frac{1}{\Gamma(\alpha)}\int_\eta^{Q_3} (Q_3-s)^{\alpha-1}|f(s,x_s)-f(s,y_s)|ds\\
\nonumber
&\leq \frac{1}{\Gamma(\alpha)}\int_\eta^{Q_3} (Q_3-s)^{\alpha-1}\frac{b_2}{(s-\eta)^{a_2}}\|x_s-y_s\|_{C[-h,0]}ds\\
\nonumber
&\leq \frac{b_2}{\Gamma(\alpha)}\int_0^{Q_3-l} (Q_3-l-s)^{\alpha-1} \frac{1}{s^{a_2}}W_3E_{1,1-a_2}(\lambda_3s)ds\\
\nonumber
&< W_3E_{1,1-a_2}(\lambda_3(Q_3-\eta)).
\end{align}
This contradiction shows that the equation  (\ref{governing_equation})-(\ref{initial_condition}) has a unique solution.
\end{proof}

\begin{remark}
The condition (H4-1) of Theorem \ref{uniqueness_theorem} can be replaced by the following condition.
There exist  $ n\in N $, $0=T_0<T_1<\cdots<T_n=T $, $ \eta_i \in (T_{i-1},T_i)$ for $i=2,\cdots,n $, $a_j \in [0,\alpha)$, $ b_j>0 $ for $ j=1,\cdots,n $ such that
\begin{equation}
\nonumber
 |f(t,u)-f(t,v)| \leq \left\{
                \begin{aligned}
                  & \frac{b_1}{t^{a_1}}\|u-v\|_{C[-h,0]} && \text{for $ t \in (0, T_1] $ and $ u,v \in C[-h,0] $} \\
                  & \frac{b_i}{|t-\eta_i|^{a_i}}\|u-v\|_{C[-h,0]} && \text{for $ t \in [T_{i-1},\eta_i) \cup(\eta_i, T_i] $}\\
                  & && \text{ and $ u,v \in C[-h,0] $},\\
                \end{aligned}
              \right.
\end{equation}
where  $ i=2,\cdots,n $.
\end{remark}

\begin{remark}
Theorem \ref{existence_theorem} is an improvement on the result of Theorem 5.1 in \cite{YangCao}.
\end{remark}

\begin{corollary}
\label{global_uniqueness}
Suppose that (H2-1) and (H4-1) hold, except that the number $ T $ is taken to be $ \infty $.
Then the initial value problem (\ref{governing_equation})-(\ref{initial_condition}) has a unique solution on $ [0, \infty)$.
\end{corollary}
\begin{proof}
Similar to Corollary \ref{global_existence}, we can prove this result.
\end{proof}

%------------------------------------------------------------------------------------------
\section*{Acknowledgement(s)}

The authors would like to thank refrees for their valuable advices for the improvement of this article.

\end{document}